\documentclass[12pt]{article}
\usepackage[reqno]{amsmath} 
\usepackage{amssymb,amsthm,verbatim,mathdots}%,multirow} %enumerate,mathrsfs

\numberwithin{equation}{section}

\newtheoremstyle{plainsl}%
	{\topsep}
	{\topsep}
	{\slshape} % only non-default setting
	{}
	{\normalfont\bfseries}
	{.}
	{ }
	{}

\theoremstyle{plainsl}
\newtheorem{theorem}{Theorem}[section]
\newtheorem{lemma}[theorem]{Lemma}
\newtheorem{corollary}[theorem]{Corollary}

\renewenvironment{proof}{\noindent{{\sl Proof. }}}{\medbreak}
\def\sqr#1#2{{\vbox{\hrule height.#2pt
    \hbox{\vrule width.#2pt height#1pt \kern#1pt
        \vrule width.#2pt}\hrule height.#2pt}}}
\def\eqed{\tag*{\sqr53}}
\def\qed{%
    \ifmmode\eqno\eqed
    \else\nobreak\ \hfill\sqr53\medbreak\fi}

\newcommand\defn[1]{\textsl{#1}}

\newcommand\al{\alpha}
\newcommand\be{\beta}
\newcommand\ga{\gamma}
\newcommand\la{\lambda}
\renewcommand\th{\theta} %--- Latex uses \th for a Norse character
\newcommand\de{\delta}
\newcommand\ka{\kappa}
\newcommand\om{\omega}
\newcommand\sg{\sigma}

\newcommand\ints{{\mathbb Z}}
\newcommand\sbs{\subseteq}
\newcommand\cx{{\mathbb C}}% complexes
\newcommand\re{{\mathbb R}}%reals
\newcommand\gras[2]{\mathcal{G}_{#1,#2}}
\newcommand\ff[1]{{\mathbb F}_{#1}}
\newcommand\zz[1]{{\mathbb Z}_{#1}}

\newcommand\ip[2]{\left\langle#1,#2\right\rangle}

\newcommand\floor[1]{\left\lfloor#1\right\rfloor}
\newcommand\smallsbinom[2]{\left[\begin{smallmatrix}#1\\#2\end{smallmatrix}\right]}
\newcommand\sbinom[2]{\left[\begin{matrix}#1\\#2\end{matrix}\right]}
\newcommand\scr[1]{{\mathcal#1}}
\newcommand\abs[1]{\left|#1\right|}
\newcommand\conj[1]{\overline{#1}}

\newcommand\opk[1]{\mathop{\mathrm{#1}}\nolimits}

\newcommand\tr{\opk{tr}}
\newcommand\diag{\opk{diag}}
\newcommand\Harm{\opk{Harm}}
\newcommand\len{\opk{len}}
\newcommand\ann{\opk{ann}}
\newcommand\Hom{\opk{Hom}}

\begin{document}

\title{Bounds for codes and designs in complex subspaces}
\author{Aidan Roy\footnote{email:aroy@qis.ucalgary.ca} \\
Institute for Quantum Information Science, University of Calgary\\
Calgary, Alberta T2N 1N4, Canada}

\maketitle

\abstract{We introduce the concepts of complex Grassmannian codes and designs. Let $\gras{m}{n}$ denote the set of $m$-dimensional subspaces of $\cx^n$: then a \defn{code} is a finite subset of $\gras{m}{n}$ in which few distances occur, while a \defn{design} is a finite subset of $\gras{m}{n}$ that polynomially approximates the entire set. Using Delsarte's linear programming techniques, we find upper bounds for the size of a code and lower bounds for the size of a design, and we show that association schemes can occur when the bounds are tight. These results are motivated by the bounds for real subspaces recently found by Bachoc, Coulangeon and Nebe, and the bounds generalize those of Delsarte, Goethals and Seidel for codes and designs on the complex unit sphere.}

% Subject class: 94B65,94B27

\section{Introduction}

In this paper, we introduce the concept of complex Grassmannian codes and designs: codes and designs in the collection of fixed-rank subspaces of a complex vector space.

In the 1970's, Delsarte \cite{del1} developed a series of excellent bounds for certain error-correcting codes by treating codewords as points in an association scheme and then applying linear programming. Shortly thereafter, Delsarte, Goethals and Seidel \cite{dgs} showed that the same technique could also be used on systems of points on the real or complex unit sphere, which they called spherical codes and spherical designs; this resulted in important contributions to problems in sphere-packing \cite[Chapter 9]{cs1}. This linear programming technique, which is now known as ``Delsarte LP theory", has proved surprisingly portable. Recently, Bachoc, Coulangeon and Nebe \cite{bcn1} generalized the results of Delsarte, Goethals and Seidel to real Grassmannian spaces, and Bachoc \cite{bac1} pointed out that ``the same game" can be played over the complex numbers. In this paper, we investigate more closely the case of complex Grassmannian codes.

The motivation for studying complex Grassmannians comes from the theory of quantum measurements. Roughly speaking, any complex Grassmannian $1$-design defines a projective measurement in the theory of quantum mechanics. It has recently been discovered that complex projective $2$-designs correspond to quantum measurements that are optimal for the purposes of nonadaptive quantum state tomography \cite{sco}. In fact, this is also true in the more general Grassmannian setting: complex Grassmannian $2$-designs are the optimal choices of measurements for nonadaptive quantum state tomography when the observer only has access to measurements with a restricted number of outcomes. More details will appear in a paper by Godsil, R\"otteler, and the author \cite{grr}. Complex Grassmannians also play a role in certain wireless communication protocols \cite{aru1}.

Define $\gras{m}{n}$ to be the set of $m$-dimensional subspaces of an $n$-dimensional complex vector space. Without loss of generality, we will always assume $m \leq n/2$. Usually, we will represent a subspace $a$ by its $n \times n$ projection matrix $P_a$. The inner product on $\gras{m}{n}$ is the trace inner product for projection matrices:
\begin{align*}
\ip{a}{b} & := \tr(P_a^*P_b) \\
& = \tr(P_aP_b).
\end{align*}
Since $\ip{a}{b} = \ip{b}{a}$, the inner product is real. This is a measure of separation, or distance, between two subspaces---note that is not a distance metric per se: the inner product of $P_a$ with itself is maximal rather than minimal. However, the \defn{chordal distance} \cite{chs}, defined by 
\[
d_c(P_a,P_b) := \sqrt{m - \tr(P_aP_b)},
\]
is a monotonic function of the inner product. Given a finite set of inner product values $\scr{A}$, an \defn{$\scr{A}$-code} is a subset $S$ of $\gras{m}{n}$ such that 
\[
\scr{A} = \{\tr(P_aP_b): a,b \in S, \; a \neq b \}.
\]
An \defn{$s$-distance set} is an $\scr{A}$-code with $|\scr{A}| = s$. This generalizes the concept of an $s$-distance set on the complex unit sphere: if $u$ and $v$ are unit vectors, then their separation distance on the unit sphere is a function of 
\[
\abs{u^*v}^2  = \tr(uu^*vv^*).
\]
We are interested in codes of maximal size for a fixed $\scr{A}$ or $s$, and bounds on their size based on zonal polynomials. Table \ref{tab:codebounds} in Section \ref{sec:bounds} gives a summary of the bounds for small $\abs{\scr{A}}$.

The outline of this paper is as follows. In Section \ref{sec:orbitals}, we describe the orbits of pairs of subspaces in $\gras{m}{n}$ under the action of $U(n)$: these orbits play a significant role in the bounds derived later on. In Sections \ref{sec:representations}, \ref{sec:symmetric} and \ref{sec:zonal}, we develop the necessary representation theory background needed for our LP bounds. In particular, we discuss the decomposition of the square-integrable functions on $\gras{m}{n}$ into irreducible representations of $U(d)$, and the zonal polynomials for these representations. The results in this section are all known, and the development is quite similar to that of Bachoc, Coulangeon and Nebe for real Grassmannians. In fact, the complex case is actually easier than the real case, because representations of the unitary group $U(n)$ are easier to describe than representations of the orthogonal group $O(n)$. In Section \ref{sec:bounds}, we develop absolute and relative bounds for codes, and show how these bounds for $\gras{m}{n}$ reduce to known bounds for complex spherical codes when $m = 1$. These bounds are compared to some other known bounds for subspaces in Section \ref{sec:otherbounds}. In Section \ref{sec:designs}, we consider Grassmanian designs. Grassmannian codes enjoy a form of duality with complex Grassmannian designs, very similar to real Grassmannian codes or spherical codes. In Section \ref{sec:examples}, we give examples in which the bounds are tight. In many cases codes of maximal size or designs of minimal size have the structure of an association scheme, which we describe in Section \ref{sec:assoc}.

\section{Orbitals}
\label{sec:orbitals}

In this section we describe the orbits of pairs of elements of $\gras{m}{n}$ under the action of $U(n)$. 

First, we claim that $\gras{m}{n}$ can be identified with a factor group of the unitary group, $\nobreak{U(n)/(U(m) \times U(n-m))}$. For, consider the first $m$ columns of a matrix of $U(n)$ as the basis for a subspace $a$ of dimension $m$ in $\cx^n$, letting the last $n-m$ columns be a basis for $a^{\perp}$. Then $a$ is invariant under the action of $U(m)$ on the first $m$ columns, while $a^{\perp}$ is invariant under $U(n-m)$. 

As a result of this factor group, $U(n)$ acts on $\gras{m}{n}$ as follows: if $U$ is in $U(n)$ and $P_a$ is the projection matrix for $a \in \gras{m}{n}$, then 
\[
U: P_a \mapsto UP_aU^*.
\]
This action is an isometry, in that it preserves the trace inner product on $\gras{m}{n}$. Unlike the complex unit sphere, however, $U(n)$ is not \defn{$2$-homogeneous} on $\gras{m}{n}$: $U(n)$ does not act transitively on pairs of subspaces with the same distance. In other words, the fact that $\tr(P_aP_b) = \tr(P_cP_d)$ does not imply that there is a unitary matrix mapping $a$ to $c$ and $b$ to $d$. In order to use zonal polynomials, we need to understand the orbits of pairs in $\gras{m}{n}$ under this isometry group, which requires principal angles. 

Given $a$ and $b$ in $\gras{m}{n}$, the \defn{principal angles} $\th_1,\ldots,\th_m$ between $a$ and $b$ are defined as follows: firstly, $\th_1$ is the largest angle that occurs between any two unit vectors $a_1 \in a$ and $b_1 \in b$:
\[
\th_1 := \min_{\substack{a_1 \in a\\b_1 \in b}} \; \arccos \abs{a_1^*b_1}.
\]
Secondly, $\th_2$ is the largest angle that occurs between any two unit vectors $a_2 \in a \cap a_1^\perp$ and $b_2 \in b \cap b_1^\perp$. Similarly define $\th_3,\ldots,\th_m$. These principle angles are closely related to the eigenvalues of $P_aP_b$: the first $m$ eigenvalues of $P_aP_b$ are $\{\cos^2 \th_1,\ldots,\cos^2\th_m\}$. Because of this correspondence, for the remainder of this paper we simply refer to the eigenvalues $y_i := \cos^2 \th_i$ (rather than the values $\th_i$) as the principal angles between $a$ and $b$. Note that $n-m$ of the eigenvalues of $P_aP_b$ are zero, so we need only consider the first $m$ eigenvalues. Conway, Hardin, and Sloane \cite{chs} accredit the following lemma to Wong \cite[Theorem 2]{wong}.

\begin{comment}
Let $M_a$ be the  $n \times m$ matrix with columns $[a_1, \ldots, a_m]$ and $M_b := [b_1,\ldots,b_m]$, where $\cos \th_i = \abs{a_i^*b_i}$. then $M_a^*M_b$ is a diagonal matrix with diagonal $(a_1^*b_1,\ldots,a_n^*b_n)$. Likewise, $M_b^*M_a$ is diagonal with entries $(b_1^*b_1,\ldots,b_n^*b_a)$, and so $M_a^*M_bM_b^*M_a$ is diagonal with diagonal entries $(\cos^2 \th_1,\ldots, \cos^2 \th_m)$. However, $M_a^*M_bM_b^*M_a$ and $M_aM_a^*M_bM_b^* = P_aP_b$ have the same nonzero eigenvalues, so the first $m$ eigenvalues of $P_aP_b$ are $\{\cos^2 \th_1,\ldots, \cos^2 \th_m\}$. \qed
\end{comment}

\begin{lemma}
The principal angles characterize the orbits of pairs of subspaces under $U(n)$.
\end{lemma}

\begin{proof}
Suppose $U \in U(n)$ maps projection matrices $P_a$ and $P_b$ to $P_c$ and $P_d$ respectively. Then by similarity, the eigenvalues of 
\[
P_cP_d = (UP_aU^*)(UP_bU^*) = UP_aP_bU^*
\]
are the same as the eigenvalues of $P_aP_b$.

Conversely, we show that if $P_aP_b$ and $P_cP_d$ have the same eigenvalues, then some unitary matrix $U$ maps $a$ to $c$ and $b$ to $d$. We do this by unitarily mapping $a$ and $b$ into a canonical form that depends only on the eigenvalues of $P_aP_b$. 

Let $M_a$ be an $n \times m$ matrix whose columns $[a_1, \ldots, a_m]$ are an orthonormal basis for $a$, so that $M_aM_a^* = P_a$ and $M_a^*M_a = I$. Similarly define $M_b = [b_1, \ldots, b_m]$ for $b$. Suppose $M_a^*M_b$ has singular value decomposition $UDV^*$, where $U$ and $V$ are $m \times m$ unitary and $D$ is $m \times m$ diagonal. Then $(M_aU)^*(M_bV) = D$. Since the columns of $M_aU$ are another orthonormal basis for $a$, without loss of generality we replace $M_a$ by $M_aU$ and likewise replace $M_b$ with $M_bV$. In other words, we may assume without loss of generality that $M_a^*M_b = D$, where $D$ is a diagonal matrix of singular values. 

Next, define the columns of $N_a = [a_{m+1},\ldots,a_n]$ to be any orthonormal basis for $a^\perp$, so that $N_aN_a^* = I-P_a$ and $N_a^*N_a = I$. Further assume that $N_a^*M_b = QR$, where $Q$ is $(n-m) \times (n-m)$ unitary and $R$ is $(n-m) \times m$ upper triangular (the $QR$-decomposition of $N_a^*M_b$). Then $Q^*N_a^*M_b = R$, and the columns of $N_aQ$ form another orthonormal basis for $a^\perp$. Replacing $N_a$ by $N_aQ$, we may assume without loss of generality that $N_a^*M_b$ is upper triangular. 

Finally, let $U_a := \Big(\begin{smallmatrix}M_a^* \\ N_a^*\end{smallmatrix}\Big)$; this is an $n \times n$ unitary matrix. Then
\[
U_aM_a = \left(\begin{matrix}I_m \\ 0 \end{matrix}\right); \quad U_aM_b  = \left(\begin{matrix}D \\ R \end{matrix}\right).
\]
If $P_aP_b$ has eigenvalues $\cos^2 \th_i$, then $M_a^*M_b = D$ has singular values $\cos \th_i$. Moreover, since $U_aM_b$ has orthonormal columns, it follows that $R$ also has orthogonal columns. We may therefore assume that $R$ is not just the upper triangular but diagonal, with diagonal entries $\sin \th_i$. Thus $U_a$ is a unitary matrix which maps $M_a$ and $M_b$ into the form
\[
M_a \mapsto \left(\begin{matrix}I_m \\ 0 \end{matrix}\right), \quad M_b \mapsto \left(\begin{matrix}
\cos \th_1 & & \\
& \ddots & \\
& & \cos \th_m \\
\sin \th_1 & & \\
& \ddots & \\
& & \sin \th_m \\
& 0 & \\
\end{matrix}\right).
\]
Since any pair $(M_a,M_b)$ with principal angles $\cos^2 \th_i$ can be mapped to this canonical form, it follows that the eigenvalues of $P_aP_b$ characterize the orbits of pairs $(a,b)$ under the unitary group. \qed
\end{proof}

\section{Representations}
\label{sec:representations}

In this section and the next, we develop the representation theory needed for Grassmannian LP bounds.

As is standard for compact Lie groups, we work with functions on $\gras{m}{n}$ to find irreducible representations. Define an inner product for functions on $\gras{m}{n}$ as follows:
\[
\ip{f}{g} := \int_{\gras{m}{n}} \conj{f(a)}g(a) \; da.
\]
Here $da$ is the unique measure invariant on $\gras{m}{n}$, normalized so that $\nobreak{\int da = 1}$. That such a measure exists an is unique (the Haar measure) follows from the fact that $\gras{m}{n}$ is a compact Lie group. Equivalently, we may write
\[
\ip{f}{g} := \int_{U(n)} \conj{f(U^*P_aU)}g(U^*P_aU) \;dU,
\]
where $dU$ is the Haar measure on $U(n)$, and $P_a$ is the projection matrix for some fixed $a \in \gras{m}{n}$. Now let $L^2(\gras{m}{n})$ denote the space of square-integrable functions on $\gras{m}{n}$. Then $U(n)$ acts on $f\ \in \gras{m}{n}$ as follows:
\[
(Uf)(P_a) := f(U^*P_aU).
\]
It follows that $L^2(\gras{m}{n})$ provides a representation of $U(n)$. As we will see, this representation can be decomposed into irreducible subrepresentations explicity, and the decomposition is \defn{multiplicity-free}: no irreducible representation of $U(n)$ occurs more than once in $L^2(\gras{m}{n})$. 

Since $U(n)$ is a compact Lie group, its irreducible representations are well-studied: see for example \cite{sep,gw1,bump,fh1}. Every irreducible representation is indexed by a \defn{dominant weight} \cite[Theorem 7.34]{sep}. In the case of $U(n)$, we may take these weights to have the form \cite[Theorem 38.3]{bump}
\[
\la = (\la_1,\ldots,\la_n): \la_1 \geq \la_2 \geq \ldots \geq \la_n, \la_i \in \ints.
\]
% Another way to treat this is $U(n) = SU(n) \times S^1/\ints_n$. Then every irreducible representation $V$ of $U(n)$ has the form $V_1 \otimes V_2$, where $V_1$ is an irreducible representation of $SU(n)$ and $V_2$ is an irreducible representation of $S^1/\ints_n$ (see p. 267 of Brocker and tomDieck). Moreover, $S^1/\ints_n$ is Abelian and $SU(n)$ is semisimple, so those irreps are "easy".
The dimension of the irreducible representation $V_\la$ indexed by $\la$ is given by \defn{Weyl's character formula} \cite[Theorem 7.32]{sep}. In the case of $U(n)$, the formula reduces to: 
\begin{equation}
\dim V_\la = \prod_{1 \leq i < j \leq n} \frac{\la_i - \la_j + j-i}{j-i}.
\label{eqn:weylchar}
\end{equation}
For example, the standard representation of $U(n)$ is indexed by $\la = (1,0,\ldots,0)$, which gives
\[
\dim V_{(1,0,\ldots,0)} = n.
\]
Note that there is more than one irreducible representation with the same dimension.

Each dominant weight may also be thought of as a form acting on a maximal Abelian subgroup of the Lie group. Here $\la$ acts on the diagonal matrix $d =\diag(d_1,\ldots,d_n) \in U(n)$ as follows:
\[
d^{\la} := \prod_{i=1}^n d_i^{\la_i}.
\]
The next section describes exactly which of these forms contribute to the decomposition of $L^2(\gras{m}{n})$.

\section{Symmetric spaces}
\label{sec:symmetric}

The group $U(n)/U(m) \times U(n-m)$ is an example of a \defn{symmetric space}: a factor group $G/K$ such that $G$ is a connected semisimple Lie group and $K$ is the fixed point set of an involutive automorphism of $G$. In this section, we use results from Goodman and Wallach \cite{gw1} to explain how the decomposition of representations of $\gras{m}{n}$ follows from this structure.

Let $s_m$ denote the $m \times m$ matrix with backwards diagonal entries of $1$ and $0$ elsewhere:
\[
s_m := \left(\begin{matrix} 
0 & & 1 \\
& \iddots & \\
1 & & 0
                 \end{matrix} \right).
\]

Then $U(n,s_n)$ denotes the group of matrices which preserve the Hermitian form $(x,y) \mapsto x^*s_ny$: that is, $U(n,s_n)$ is the set of matrices $M$ such that $M^*s_nM = s_n$. This group is isomorphic the standard unitary group $U(n)$. Define
\[
J_{m,n} := \left(\begin{matrix} 
& & s_m \\
& I_{n-2m} & \\
s_m & &
                 \end{matrix} \right),
\]
and consider the involution $\th(M) := J_{m,n}MJ_{m,n}$ on $GL_n(\cx)$. The fixed points of $\th$ have the form 
\[
M = \left(\begin{matrix} 
a & b & c \\
d & e & ds_m \\
s_mcs_m & s_mb & s_mas_m 
                 \end{matrix} \right),
\]
so the fixed point set in $GL_n(\cx)$ is isomorphic to $GL_m(\cx) \times GL_{n-m}(\cx)$. 

\begin{lemma}
The fixed point set $K$ of $\th$ in $G = U(n,s_n)$ is isomorphic to $U(m) \times U(n-m)$. Therefore $\gras{m}{n}$ is a symmetric space.
\end{lemma}

\begin{proof}
For $a = (a_1,\ldots,a_m)$, let $\breve{a}$ denote the reversal of $a$, namely
\[
\breve{a} := s_ma = (a_m,\ldots,a_1).
\]
If $a$, $b$, and $c$ have length $m$, $n-2m$ and $m$ respectively, then we have $J_{m,n}(a,b,c)^T = (\breve{c},b,\breve{a})^T$. Therefore the $1$ and $-1$ eigenspaces of $J_{m,n}$ are $V_+ = \{(a,b,\breve{a})\}$ and $V_- = \{(a,0,-\breve{a})\}$ respectively. These spaces are orthogonal with respect to the form $(x,y) \mapsto x^*s_ny$. 

Now $K$ is the set of points in $U(n,s_n)$ which commute with $J_{m,n}$. So decomposing $\cx^n$ into $V_+ \oplus V_-$, we have that $K$ is the set of points in $U(n,s_n)$ which leave both $V_+$ and $V_-$ invariant. In other words, $K$ is the set of points which preserve the form $s_n$ on the subspaces $V_+$ and $V_-$. Thus 
\[
K \cong U(V_+,s_n|_{V_+}) \times U(V_-,s_n|_{V_-}) \cong U(n-m) \times U(m). \eqed
\]
\end{proof}

\begin{comment}
Let $P_{\pm}$ project on to $V_{\pm}$, so $J' = V_+ - V_-$. Then $K$ is the set of points in $U(n,s_n)$ that commute with both $V_+$ and $V_-$. Now $U(V_+,s_n|_{V_+})$ is the set of matrices that $(P_+Ma)^*s_n(P_+Mb) = a^*s_nb$ for all $a,b \in V_+$: that is, the matrices satisfying $P_+M^*P_+s_nP_+MP_+ = P_+s_nP_+$,  and similarly for $V_-$. Any $M \in U(n,s_n)$ commuting with $P_+$ satisfies this condition. Conversely, since $P_{\pm}$ commutes with $s_n$ and $P_\pm s_nP_\pm = \pm P_\pm$, we have that $U(V_\pm,s_n|_{V_\pm})$ is the set of $M$ satisfying $P_\pm M^*P_\pm MP_\pm = P_\pm$.  
\end{comment}

The fact that $K$ is the fixed point set of $\th$ in $G$ implies (\cite[Theorem 12.3.5]{gw1}) that $(G,K)$ is a \defn{spherical pair}: for every irreducible representation $V_\la$ of $G$, the subspace $V_\la^K$ of points fixed by $K$ satisfies $\dim V_\la^K \leq 1$. Those representations such that $V_\la^K$ has dimension exactly $1$ are called \defn{spherical representations}. The following theorem \cite[Theorem V.4.3]{hel} explains how those representation relate to $L^2(G/K)$.

\begin{theorem}
Let $G$ be a compact simply connected semisimple Lie group, and let $K \leq G$ be the fixed point group of an involutive automorphism of $G$. Further let $\hat{G}_K$ denote the set of equivalence classes of spherical representations $V_\la$ of $G$ with respect to $K$. Then $L^2(G/K)$ is a multiplicity-free representation of $G$, and
\[
L^2(G/K) \cong \bigoplus_{\la \in \hat{G}_K} V_\la.
\]
\end{theorem}
% This is in fact a Hilbert space decomposition.

To describe which representations are spherical, we now consider diagonal subgroups of $G$ and $K$. For $d = (d_1,\ldots,d_n)$, let $\diag(d)$ denote the diagonal matrix with diagonal entries $d_1,\ldots,d_n$. Firstly, note that $\diag(d)$ is in $U(n,s_n)$ if and only if $d_{n+1-k} = 1/\bar{d_k}$, where $\bar{d_k}$ is the complex conjugate of $d_k$. In other words, if $\bar{d}^{-1}$ denotes the vector $(1/\bar{d_1},\ldots,1/\bar{d_k})$, then $\diag(d)$ is in $U(n,s_n)$ if and only if $\breve{d} = \bar{d}^{-1}$. Secondly, note that if $d = \diag(a,b,c)$ with $a$ and $c$ of length $m$, then $\th(d) = (\breve{c},b,\breve{a})$. It follows that the diagonal group 
\[
T := \{\diag(a_1,\ldots,a_m,c_{m+1},\ldots,c_{n-m},a_m,\ldots,a_1): \abs{a_i} = 1, \breve{c} = \bar{c}^{\,-1}\}
\]
is contained in $K$. In fact, it is a maximal Abelian subgroup of $K$: this is called a \defn{torus} of $K$. 

\begin{comment}
On the other hand, $A = \{\diag(a,c,\breve{c},\breve{a}^{-1}): c_i = \pm 1, a_i \in \re\}$ is a maximally Abelian subgroup of $G$ such that $\th(h) = h^{-1})$ for every $h \in A$: this is called the \defn{anisotropic} torus. 
\end{comment}

Recall that the irreducible representations of $G$ are indexed by the dominant weights $\la = (\la_1,\ldots,\la_n)$, where $\la_i \geq \la_{i+1}$ and $\la_i \in \ints$. Now the spherical representations of $G$ with respect to $K$ are indexed by those particular dominant weights such that $t^{\la} = 1$ for all $t = (t_1,\ldots,t_n)$ in the torus $T$ (see Goodman and Wallach \cite[p. 540]{gw1}). So a dominant weight $\la$ is spherical if it has the form
\[
\la = (\la_1,\ldots,\la_m,0,\ldots,0,-\la_m,\ldots,-\la_1)
\]
with $\la_1 \geq \ldots \geq \la_m \geq 0$ and $\la_i \in \ints$. In other words:

\begin{theorem}
The irreducible representations of $U(n)$ occurring in $L^2(\gras{m}{n})$ are in one-to-one correspondence with the integer partitions with at most $m$ parts.
\end{theorem}

For any partition $\mu$, we let $H_\mu(n)$, or simply $H_\mu$, denote the irreducible representation in $L^2(\gras{m}{n})$ isomorphic to $V_{(\mu,0,\ldots,0,-\breve{\mu})}$. The Weyl character formula (equation \eqref{eqn:weylchar}) now tells us the dimension of each $H_\mu$. The first few dimensions are:
\begin{align*}
\dim H_{(0)} = \dim V_{(0,\ldots,0)} & = 1 \\
\dim H_{(1)} = \dim V_{(1,0,\ldots,0,-1)}& = n^2 - 1 \\
\dim H_{(2)} & = \frac{n^2(n-1)(n+3)}{4} \\
\dim H_{(1,1)} & = \frac{n^2(n+1)(n-3)}{4} \\
\dim H_{(2,1)} & = \frac{(n^2-1)^2(n^2-9)}{9} \\
\dim H_{(k)} & = \binom{n+k-2}{k}^2\frac{n+2k-1}{n-1} \\
\dim H_{(\underbrace{\scriptstyle{1,\ldots,1}}_k)} & = \binom{n+1}{k}^2\frac{n-2k+1}{n+1} 
\end{align*}

If $m = 1$, then $\gras{m}{n}$ is the complex projective space $\cx P^{n-1}$, and only the spaces $H_{(k)}$ occur. In that case $H_{(k)}$ is isomorphic to the space $\Harm(k,k)$ of harmonic polynomials of homogeneous degree $k$ in both $z$ and $\bar{z}$, where $z = (z_1,\ldots,z_n)$ is a point on the unit sphere in $\cx^n$. Those harmonic polynomials were used by Delsarte, Goethals, and Seidel in their LP bounds for codes and designs the complex unit sphere \cite{dgs}.

We now record a few more representations of $U(n)$ we will need later. Given an nonincreasing sequence of nonnegative integers $\mu = (\mu_1,\mu_2,\ldots)$, we say $\mu$ has \defn{size} $k$ and write $\abs{\mu} = k$ if $\mu$ is a partition of $k$; that is, $\sum_i \mu_i = k$. We also say $\mu$ has \defn{length} $l$ and write $\len(\mu) = l$ if $\mu$ has $l$ nonzero entries. For example, $(2,1,0,\ldots)$ has size $3$ and length $2$. Then for fixed $\gras{m}{n}$, define $H_k = H_k(m,n)$ as follows:
\[
H_k(m,n) := \bigoplus_{\substack{\abs{\mu} \leq k \\ \len(\mu) \leq m}} H_\mu(n).
\]
For $k > 0$ this representation is reducible, and $H_{k-1}$ is contained in $H_k$. When $m=1$, $H_k$ is isomorphic to the space of homogeneous polynomials degree $k$ in both $z$ and $\bar{z}$ on the unit sphere in $\cx^n$. In the next section, we will see that $H_k$ is also the span of the
degree-$k$ symmetric polynomials on the principal angles between $a \in \gras{m}{n}$ and some fixed $b \in \gras{m}{n}$. Moreover, if $g$ and $h$ are polynomials in $H_k$ and $H_{k'}$ respectively, then $gh$ is in $H_{k+k'}$, and in fact $H_{k+k'}$ is spanned by polynomials of that form.

We also let $\Hom_k(n) \sbs L^2(\gras{m}{n})$ denote the space of polynomials which are homogeneous of degree $k$ in the entries of $P_a$, where $P_a$ is the projection matrix of $a \in \gras{m}{n}$. Since the constant function $P_a \mapsto \tr(P_a) = m$ is in $\Hom_1(n)$, it follows that $\Hom_{k-1}(n)$ can be embedded into $\Hom_k(n)$. Similarly for fixed $b$, the distance function $P_a \mapsto \tr(P_aP_b)$ is in $\Hom_1(n)$. The next section will also show that $H_k$ is a subspace of $\Hom_k$.

James and Constantine \cite{jc1} further investigated the irreducible subspaces of $L^2(\gras{m}{n})$, finding zonal polynomials for each irreducible representation. We describe those results in the next section.

\section{Zonal polynomials}
\label{sec:zonal}

A \defn{zonal polynomial} at a point $a \in \gras{m}{n}$ is a function on points $b \in \gras{m}{n}$ which depends only on the the principle angles between $a$ and $b$. Given any univariate polynomial $f(x)$ of degree $k$, we define the \defn{zonal polynomial of $f$ at $b$} as follows: if $f(x) = \sum_{i=0}^k f_i x^i$, then
\[
f_a(b) = \sum_{i=0}^k f_i \tr(P_aP_b)^i.
\]
Here $P_a$ and $P_b$ are the projection matrices for $a$ and $b$. As written, the zonal polynomial is not homogeneous, but by embedding the constant $1$ into $\Hom_1(n)$ in the form $\tr(P_b)/m$, the exponents in $f_a(b)$ may be ``pushed up'' and we may assume $f_a$ is in $\Hom_k(n)$. To see that $f_a(b)$ only depends on the principal angles between $a$ and $b$, note that $\tr(P_aP_b)$ is simply the sum of the principal angles. 

There is another set of zonal polynomials that play a particular role in the theory of Delsarte bounds. Let $H_\mu$ be an irreducible representation in $L^2(\gras{m}{n})$. Then for each $a \in \gras{m}{n}$, define the \defn{zonal orthogonal polynomial} $Z_{\mu,a}$ to be the unique element of $H_\mu$ such that for every $p \in H_\mu$,
\[
\ip{Z_{\mu,a}}{p} = p(a).
\]
Then zonal polynomials are invariant under the unitary group, in the following sense:
\[
Z_{\mu,b}(a) = \ip{U^*Z_{\mu,a}}{U^*Z_{\mu,b}} = \ip{Z_{\mu,Ua}}{Z_{\mu,Ub}} = Z_{\mu,Ub}(Ua).
\]
The value of $Z_{\mu,b}(a)$ depends on the $U(n)$-orbit of $(a,b)$ and therefore depends on the principle angles of $a$ and $b$. With this in mind we sometimes write $Z_{\mu,a}(b) = Z_\mu(a,b)$ or $Z_{\mu,a}(b) = Z_\mu(y_1,\ldots,y_m)$, where $(y_1,\ldots,y_m)$ are the principal angles of $a$ and $b$.

Schur orthogonality \cite[Theorem 3.3]{sep} for irreducible representations implies that $Z_{\mu,a}$ and $Z_{\nu,b}$ are orthogonal for $\mu \neq \nu$. So, we have
\[
\ip{Z_{\mu,a}}{Z_{\nu,b}} = \de_{\mu,\nu} Z_{\mu}(a,b). 
\]
Moreover, $Z_{\mu,a}(b) = Z_{\mu,b}(a)$ is in fact real and symmetric in $a$ and $b$. The zonal polynomials satisfy some other important properties, including the following positivity condition:

\begin{lemma}
\label{lem:possum}
For any subset $S \sbs \gras{m}{n}$,
\[
\sum_{a,b \in S} Z_\mu(a,b) \geq 0.
\]
Equality holds only when $\sum_{a \in S} Z_{\mu,a} = 0$.
\end{lemma}

\begin{proof}
We have
\begin{align*}
\sum_{a,b \in S} Z_\mu(a,b) & = \sum_{a,b \in S} \ip{Z_{\mu,a}}{Z_{\mu,b}} \\
& = \ip{\sum_{a \in S} Z_{\mu,a}}{\sum_{b \in S} Z_{\mu,b}} \\
& \geq 0.
\end{align*} 
Equality holds if and only if $\sum_{a \in S} Z_{\mu,a} = 0$. \qed
\end{proof}

The second important condition the zonal polynomials satisfy is called the \defn{addition formula}:

\begin{lemma}
\label{lem:addform}
Let $e_1,\ldots,e_N$ be an orthonormal basis for the irreducible subspace $H_\mu$. Then
\[
\sum_{i=1}^N \conj{e_i(a)}e_i(b) = Z_{\mu}(a,b).
\]
\end{lemma}

\begin{proof}
Since $Z_{\mu,a}$ is in $H_\mu$, we may write it as a linear combination of $e_1,\ldots,e_N$:
\begin{align*}
Z_{\mu,a} & = \sum_{i=1}^N \ip{e_i}{Z_{\mu,a}}e_i \\
& = \sum_i \conj{e_i(a)} e_i.
\end{align*}
So, it follows that $Z_{\mu,a}(b) = \sum_i \conj{e_i(a)} e_i(b)$. \qed
\end{proof}

James and Constantine give an explicit formula for the zonal orthogonal polynomials of $\gras{m}{n}$ in terms of Schur polynomials, the irreducible characters of $SL(m,\cx)$. If $y = (y_1,\ldots,y_m)$ are variables and $\sg = (s_1,\ldots,s_m)$ is a partition into at most $m$ parts, then the (unnormalized) \defn{Schur polynomial} is defined as
\[
X_\sg(y) := \frac{\det(y_i^{s_j+m-j})_{i,j}}{\det(y_i^{k-j})_{i,j}}.
\]
Each Schur polynomial is a symmetric polynomial in $(y_1,\ldots,y_m)$. For more information about Schur polynomials, see Stanley \cite[Chapter 7]{stan}. The \defn{normalized Schur polynomial} $X^*_\sg$ is the multiple of $X_\sg$ such that $X^*_\sg(1,\ldots,1) = 1$.

To define the zonal orthogonal polynomials for $\gras{m}{n}$, first define the \defn{ascending product}
\[
(a)_s := a(a+1)\ldots(a+s-1),
\]
and given a partition $\sg = (s_1,\ldots,s_m)$, define \defn{complex hypergeometric coefficients}
\[
[a]_{\sg} := \prod_{i=1}^m (a-i+1)_{s_i}.
\]
Further assume we have a partial order $\leq$ on partitions defined such that $\nobreak{(s_1,\ldots,s_m) \leq (k_1,\ldots,k_l)}$ if and only if $s_i \leq k_i$ for all $i$. Letting $y+1 := (y_1+1,\ldots,y_m+1)$, the \defn{complex hypergeometric binomial coefficients} $\smallsbinom{\ka}{\sg}$ are given by the formula
\[
X^*_\ka(y+1) = \sum_{\sg \leq \ka} \sbinom{\ka}{\sg}X^*_{\sg}(y).
\]
We can now define the zonal orthogonal polynomials for $\gras{m}{n}$. The following result is due to James and Constantine \cite{jc1}.

\begin{theorem}
Let 
\[
\rho_\sg := \sum_{i=1}^m s_i(s_i-2i+1)
\]
and let $\sg$ and $\ka$ partition $s$ and $k$ respectively. Also let
\[
[c]_{(\ka,\sg)} := \sum_i\frac{\sbinom{\ka}{\sg_i}\sbinom{\sg_i}{\sg}}{(k-s)\sbinom{\ka}{\sg}}\frac{[c]_{(\ka,\sg_i)}}{\left(c + \frac{\rho_\ka - \rho_\sg}{k-s}\right)},
\]
where the summation is over partitions $\sg_i = (s_1,\ldots,s_{i-1},s_i+1,s_{i+1},\ldots)$ that are nonincreasing. Then up to normalization, the zonal orthogonal polynomial for $H_\ka$ is
\[
Z_{\ka}(y) := \sum_{\sg \leq \ka} \frac{(-1)^s \sbinom{\ka}{\sg}[c]_{(\ka,\sg)}}{[a]_\sg}X^*_\sg(y),
\]
where $y = (y_1,\ldots,y_m)$ is the set of principal angles.
\end{theorem}

The first few normalized Schur polynomials are:
\begin{align*}
X^*_0(y) & = 1 \\
X^*_1(y) & = \frac{1}{m}\sum_{i=1}^m y_i \\
X^*_{1,1}(y) & = \frac{1}{\binom{m}{2}} \sum_{i<j} y_iy_j \\
X^*_2(y) & = \frac{1}{\binom{m+1}{2}}\Big(\sum_{i=1}^m y_i^2 + \sum_{i<j}y_iy_j \Big).
\end{align*}
Up to normalization by a constant, the first few zonal orthogonal polynomials are:
\begin{align*}
Z_0(y) & = 1 \\
Z_1(y) & = nX^*_1(y) - m \\
Z_{1,1}(y) & = m(m-1) - 2(n-1)(m-1)X^*_1(y) + (n-1)(n-2)X^*_{1,1}(y) \\
Z_2(y) & = m(m+1) - 2(n+1)(m+1)X^*_1(y) + (n+1)(n+2)X^*_{2}(y).
\end{align*}
The correct normalizations satisfy 
\[
\ip{Z_{\mu,a}}{Z_{\mu,a}} = Z_\mu(1,1,\ldots,1) = \dim H_\mu.
\]
With the exception of the case $\mu = 0$ (which is normalized correctly in the formula above), normalizations for $Z_\mu$ will not play a role in the results which follow.

\section{Bounds}
\label{sec:bounds}

Recall that an $\scr{A}$-code is a collection $S$ of subspaces in $\gras{m}{n}$ such that $\tr(P_aP_b) \in \scr{A}$ for every $a \neq b$ in $S$. In this section, we find upper bounds on the size of an $\scr{A}$-code in terms of either the cardinality of $\scr{A}$ or its elements. A summary of the results for $\abs{\scr{A}} \leq 2$ is given in Table \ref{tab:codebounds}.

\begin{table}
\renewcommand{\arraystretch}{2.2}
\setlength{\tabcolsep}{10pt}
\begin{center}
\begin{tabular}{|c||c|c|}
\hline
$\scr{A}$ & $\{\al\}$ & $\{\al,\be\}$ \\
\hline
\rule{0pt}{37pt} \parbox{0.7in}{Absolute bound\\} & \raisebox{8pt}{$n^2$} & \raisebox{7pt}{$\qquad \qquad  \quad \; \dbinom{n^2}{2} \qquad (m > 1)$} \\
\hline 
\rule{0pt}{30pt} \parbox{0.7in}{Relative bound} & \raisebox{-2pt}{$\dfrac{n(m-\al)}{m^2-n\al}$} &  \raisebox{1pt}{$\dfrac{n(m-\al)(m-\be)}{m^2\left[\frac{(m+1)^2}{2(n+1)}+\frac{(m-1)^2}{2(n-1)}-(\al+\be) + \frac{n\al\be}{m^2}\right]}$} \\
\rule{0pt}{40pt} \parbox{0.7in}{Relative bound\\conditions} & $\al < \dfrac{m^2}{n}$ & $\al+\be \leq \dfrac{2(m^2n-4m+n)}{n^2-4}$, \\
\rule{0pt}{30pt} & & \raisebox{10pt}{$\al + \be - \dfrac{n\al\be}{m^2} < \dfrac{m^2n-2m+n}{n^2-1}$} \\
\hline
\end{tabular}
\caption{Upper bounds on $|S|$, when $S \sbs \gras{m}{n}$ is an $\scr{A}$-code.}
\label{tab:codebounds}
\end{center}
\end{table}

If $\scr{A} = \{\al_1,\ldots,\al_k\}$, then the \defn{annihilator} of $\scr{A}$ is the function 
\[
\ann_{\scr{A}}(x) := \prod_{i=1}^k (x-\al_i),
\]
The significance of the annihilator is that $\ann_{\scr{A}}(\tr(P_aP_b)) = 0$ for any $a \neq b$ in $S$. More generally, for any polynomial $f$, an \defn{$f$-code} is a collection $S$ of subspaces such that $f(\tr(P_a)) \neq 0$ and $f(\tr(P_aP_b)) = 0$ for every $a \neq b$ in $S$. If $\scr{A}$ is any set of angles and $f$ is the annihilator of $\scr{A}$, then an $\scr{A}$-code is also an $f$-code.

\begin{theorem}
\label{thm:abscodebnd}
If $S \sbs \gras{m}{n}$ is an $\scr{A}$-code, with $|\scr{A}| = k$, then
\[
|S| \leq \dim(\Hom_k(n)) \leq \binom{n^2+k-1}{k}.
\]
\end{theorem}

\begin{proof}
We prove more generally that if $S$ is an $f$-code, with $\deg(f) = k$, then $|S| \leq \dim(\Hom_k(n))$. The result then follows by taking $f$ to be the annihilator of $\scr{A}$. 

Consider the zonal polynomials $f_a(b) := f(\tr(P_aP_b))$, for $a \in S$. Note that $f_a$ is in $\Hom_k(n)$, since $f_a(b)$ is a degree-$k$ polynomial in the entries of $P_b$. Since $f_a(b) = 0$ for every $b \in S$ except $a$, and $f_a(a) \neq 0$, the set $\{f_a: a \in S\}$ is linearly independent. Thus the number of functions $|S|$ is at most the dimension of the space $\Hom_k(n)$. \qed
\end{proof}

\begin{corollary}
\label{cor:abscodebnd1}
Let $S$ be a collection of subspaces in $\gras{m}{n}$ such that $\tr(P_aP_b) = \al$ for all $a \neq b$ in $S$. Then
\[
|S| \leq n^2.
\]
\end{corollary}

\begin{proof}
Use Theorem \ref{thm:abscodebnd} with the degree-$1$ annihilator of $\al$, which induces zonal polynomials in $\Hom_1(n)$. \qed
\end{proof}

Since $f_a(b)$ is also a degree-$k$ symmetric polynomial in the principal angles of $a$ and $b$, it follows that $f_a$ is also in $H_k(m,n)$. Then by the same argument as in Theorem \ref{thm:abscodebnd}, we have 
\begin{corollary}
\label{cor:abscodebnd}
If $S \sbs \gras{m}{n}$ is an $\scr{A}$-code, with $|\scr{A}| = k$, then
\[
|S| \leq \dim(H_k(m,n)).
\]
\end{corollary}

If equality holds, then the functions $f_a$ form a basis for the space. Moreover, the space $H_k(m,n)$ is exactly the space of functions on $S$. 

Theorem \ref{thm:abscodebnd} and Corollary \ref{cor:abscodebnd} are called \defn{absolute bounds} for Grassmannian codes, because the bounds depend only on the number of different inner product values that occur in $S$. When $m = 1$ these bounds reduce to the absolute bounds of Delsarte, Goethals and Seidel \cite[Theorem 6.1]{dgs}. There is also a \defn{relative bound}, which depends on the actual values of the inner products and is sometimes tighter.

\begin{theorem}
\label{thm:relcodebnd}
Let $f(y_1,\ldots,y_m) \in \re[y_1,\ldots,y_m]$ be a symmetric polynomial such that $f = \sum_{\mu}c_{\mu} Z_{\mu}$, where $Z_{\mu}$ is a zonal orthogonal polynomial, and each $c_{\mu} \geq 0$. Further assume that $c_0$ is strictly positive. If $S$ is a set of subspaces in $\gras{m}{n}$ such that $f_a(b) := f(y_1(a,b),\ldots,y_m(a,b))$ is nonpositive for every $a \neq b$ in $S$, then 
\[
|S| \leq \frac{f(1,\ldots,1)}{c_0}.
\]
\end{theorem}

\begin{proof}
Since $f_a(b) \leq 0$ for $b \neq a$, summing over all $b \in S$, we have
\[
\sum_{b \in S} f_a(b) \leq f_a(a) = f(1,\ldots,1).
\]
Then averaging over all $a \in S$,
\begin{align*}
f(1,\ldots,1) & \geq \frac{1}{|S|}\sum_{a,b \in S} f_a(b) \\
& = \frac{1}{|S|} \sum_\mu c_\mu \sum_{a,b \in S} Z_\mu(a,b). 
\end{align*}
By Lemma \ref{lem:possum}, the inner sum is non-negative for $\mu \neq 0$. If $\mu = 0$, then $Z_0(a,b) = 1$ for all $a$ and $b$, and hence,
\begin{align*}
f(1) & \geq \frac{1}{|S|} c_0 \sum_{a,b \in S} 1 \\
& = c_0 |S|. \eqed
\end{align*} 
\end{proof}

Equality holds if and only if $f_a(b) = 0$ for every $a \neq b \in S$ and for each $\mu \neq 0$, we have either $c_\mu = 0$ or $\sum_{a \in S} Z_{\mu,a} = 0$. (We will see in Section \ref{sec:designs} that when $c_\mu > 0$ for all $\abs{\mu} \leq \deg(f)$, this implies that we have a Grassmannian $t$-design.)

By way of example, we consider the case of an $\{\al\}$-code in detail.

\begin{corollary}
\label{cor:onebound}
Let $S$ be a subset of $\gras{m}{n}$ such that $\tr(P_aP_b) = \al$ for all $a \neq b$ in $S$, and $\al < m^2/n$. Then
\[
|S| \leq \frac{n(m-\al)}{m^2-n\al}.
\]
\end{corollary}

\begin{proof}
The first two zonal orthogonal polynomials are $Z_0(y) = 1$ and (up to normalization) $Z_1(y) = \sum_{i=1}^my_i - m^2/n$. The annihilator for $\al$ is the polynomial $f(x) = x - \al$, which induces the zonal polynomial at subspace $a$ given by 
\[
f_{a}(b) = \tr(P_aP_b) - \frac{\al\tr(P_b)}{m} = \sum_{i=1}^m y_i - \al,
\]
for principle angles $y_1,\ldots,y_m$ of $a$ and $b$. Thus we may write
\begin{align*}
f(y_1,\ldots,y_m) & = \sum_{i=1}^m y_i - \al \\
& = c_1Z_{1}(y) + \left(\frac{m^2}{n}-\al\right)Z_{0}(y).
\end{align*}
Applying Theorem \ref{thm:relcodebnd}, we find that
\[
|S| \leq \frac{f(1,\ldots,1)}{c_0} = \frac{m-\al}{m^2/n-\al}. \eqed
\]
\end{proof}
When $m = 1$, we recover Delsarte, Goethals and Seidel's bound for complex equiangular lines: 
\[
|S| \leq \frac{n(1-\al)}{1-n\al}.
\]

Similarly, using the zonal orthogonal polynomials $Z_0,Z_1,Z_{1,1}$ and $Z_2$, we get a bound on the size of a subset containing two inner products, say $\al$ and $\be$.

\begin{corollary}
\label{cor:twobound}
Let $S$ be a subset of $\gras{m}{n}$ such that $\tr(P_aP_b) \in {\al,\be}$ for all $a \neq b$ in $S$. Further assume that
\begin{align*}
\al+\be & \leq \frac{2(m^2n-4m+n)}{n^2-4}, \\
\al + \be - \frac{n\al\be}{m^2} & < \frac{m^2n-2m+n}{n^2-1}.
\end{align*}
Then
\[
|S| \leq \frac{n(m-\al)(m-\be)}{m^2\left[\frac{(m+1)^2}{2(n+1)}+\frac{(m-1)^2}{2(n-1)}-(\al+\be) + \frac{n\al\be}{m^2}\right]}.
\]
\end{corollary}

When $m = 1$ this reduces to the Delsarte, Goethals and Seidel bound of 
\[
|S| \leq \frac{n(n+1)(1-\al)(1-\be)}{2-(n+1)(\al+\be) + n(n+1)\al\be}
\]
for lines in complex projective space $\cx P^{n-1}$.

\section{Other bounds}
\label{sec:otherbounds}

Certain cases of equality in Corollaries \ref{cor:onebound} and \ref{cor:twobound} also achieve equality for bounds on the size of the largest angle in a set of subspaces. For real Grassmannians, Conway, Hardin and Sloane \cite{chs} call these bounds the \defn{simplex} and \defn{orthoplex} bounds. Here we give their complex analogues. 

Recall that if $P_a$ be the $n \times n$ projection matrix for $a \in \gras{m}{n}$, then $P_a$ is Hermitian with trace $m$, so $P'_a = P_a - mI/n$ lies in a real space of dimension $n^2-1$. Moreover $||P'_a||^2 := \tr(P'_aP'_a) = m(1-m/n)$, so $P'_a$ is embedded onto a sphere of radius $\sqrt{m(1-m/n)}$ in $\re^{n^2-1}$. Further recall that the chordal distance on $\gras{m}{n}$ is defined by 
\begin{align*}
d_c(a,b)^2 & = m - \tr(P_aP_b) \\
& = \frac{1}{2}||P_a - P_b||^2 = \frac{1}{2}||P'_a - P'_b||^2.
\end{align*}
With this distance, the Grassmannians are isometrically embedded into $\re^{n^2-1}$. The ``Rankin bounds" given in Theorem \ref{thm:rankin} below (see \cite[Theorems 6.1.1 \& 6.1.2]{bor}) are bounds on the minimum distance between points on a real sphere as a function of the number of points and the dimension of the space. An \defn{equatorial simplex} refers to a set of $N$ points on the unit sphere that form a simplex in a hyperplane of dimension $N-1$.

\begin{theorem}
\label{thm:rankin}
Given $N$ points on a sphere of radius $r$ in $\re^D$, the minimum distance $d$ between any two points satisfies 
\[
d \leq r\sqrt{\frac{2N}{N-1}}.
\]
Equality requires $N \leq D+1$ and occurs if and only if the points form a regular equatorial simplex. For $N > D+1$, the minimum distance satisfies
\[
d \leq r\sqrt{2},
\]
and equality requires $N \leq 2D$. When $N = 2D$, equality occurs if and only if the points are the vertices of a regular orthoplex.
\end{theorem}

Conway, Hardin and Sloane \cite{chs} apply these bounds to get the simplex and orthoplex bounds for real Grassmannians: we can do the same for the complex Grassmannians.

\begin{corollary}
\label{cor:simortho}
Given a set $S$ points in $\gras{m}{n}$, the largest inner product value $\al = \ip{a}{b}$ between any two points satisfies 
\begin{equation}
\al \geq m\frac{m|S|-n}{n|S|-n}.
\label{eqn:albound}
\end{equation}
Equality requires $|S| \leq n^2$ and occurs if and only if the points form a regular equatorial simplex in $\re^{n^2-1}$. For $|S| > n^2$, the largest inner product $\be$ satisfies
\begin{equation}
\be \geq \frac{m^2}{n},
\label{eqn:albebound}
\end{equation}
and equality requires $|S| \leq 2(n^2-1)$. Equality occurs if the points are the $2(n^2-1)$ vertices of a regular orthoplex in $\re^{n^2-1}$.
\end{corollary}

If $S$ is an $\{\al\}$-code, then solving inequality \eqref{eqn:albound} for $|S|$ recovers the relative bound in Corollary \ref{cor:onebound}. Moreover, if $|S| = n^2$ (equality in the absolute bound of Corollary \ref{cor:abscodebnd1}), then
\[
\al = \frac{m(mn-1)}{n^2-1}.
\]
On the other hand, if $S$ is a $\{0,m^2/n\}$-code, and $m = n/2$, then the relative bound in Corollary \ref{cor:twobound} implies that 
\[
|S| \leq 2(n^2-1),
\]
which corresponds to equality in the orthoplex bound \eqref{eqn:albebound}.

\section{Examples}
\label{sec:examples}

In this section we give examples demonstrating the tightness of the bounds in the previous sections.

When the rank $m$ of the Grassmannian subspaces is $1$, we recover all the classical results of Delsarte, Goethals and Seidel \cite{dgs} for lines in complex projective space: their paper gives several examples of bounds with equality. In particular, the upper bound for $\{\al\}$-codes in $\cx P^{n-1}$ is $n^2$, and equality can only hold with a trace inner product value of $\al = 1/(n+1)$. Examples of tightness have been found for several small values of $n$ and are conjectured to exist for every $n$. These equiangular lines are sometimes called \defn{symmetric informationally complete POVMs} in the quantum information literature: see \cite{rbsc} for more details or \cite{kha} for recent results. Another important example in $\gras{1}{n}$ is the relative bound (Corollary \ref{cor:twobound}) with inner product values of $\al = 0$ and $\be = 1/n$. The upper bound for the size of an $\{0,1/n\}$-code is $n(n+1)$, and when equality is achieved we have what is known as a \defn{maximal set of mutually unbiased bases}. Constructions achieving the bound are known when $n$ is a prime power; see \cite{gr1} for some constructions and \cite{rs1} for applications to quantum information.

In the case $m = n/2$, if $a$ is in $\gras{m}{n}$, then its orthogonal complement $a^{\perp}$ is also in $\gras{m}{n}$, and $a$ and $a^{\perp}$ have a trace inner product of $0$. Here again, such subspaces have applications in quantum state tomography; more details will be found in \cite{grr}. If $S$ is a $\{0,n/4\}$-code in $\gras{n/2}{n}$, then by the relative bound (Corollary~\ref{cor:twobound}), $S$ has size at most $2(n^2-1)$. In these case we may assume that both $a$ and $a^{\perp}$ are in $S$, because if $a$ and $b$ have a trace inner product of $n/4$, then so do $a^{\perp}$ and $b$. The following construction, due to Martin R\"otteler, demonstrates that Corollary~\ref{cor:twobound} is tight when $n$ is a power of $2$.
\begin{theorem}
\label{thm:pauli}
Let $X_1,\ldots,X_{n^2-1}$ be the Pauli matrices of order $n = 2^k$, and let 
\[
M_i := \frac{1}{2}(I + X_i).
\]
Then $\cup_{i=1}^{n^2-1} \{M_i,I-M_i\}$ is the set of projection matrices for a $\{0,n/4\}$-code of size $2(n^2-1)$ in $\gras{n/2}{n}$. 
\end{theorem}
More generally, the bound is tight when $n$ is the order of a Hadamard matrix: details of the following construction will appear in in \cite{grr}.

\begin{theorem}
\label{thm:hadspace}
Suppose there is a Hadamard matrix of order $n$. Then there exists a $\{0,n/4\}$-code of size $2(n^2-1)$ in $\gras{n/2}{n}$. 
%set of $2(n^2-1)$ subspaces in $\gras{n/2}{n}$ with inner product values $\al = 0$, $\be = n/4$.
\end{theorem}

\begin{comment}
More generally, Zauner \cite{zau} used the affine geometry $AG(k,q)$ similar collections subspaces whenever $n$ is a prime power. His result is the following. 

\begin{lemma}
Let $n = q^k$, $m = q^{k-1}$, $\al = 0$, and $\be = q^{k-2}$, where $q$ is a prime power. Then there exists a $\{\al,\be\}$-code of size $\frac{q(q^{2k}-1)}{q-1}$ in $\gras{m}{n}$, thereby satisfying the relative bound with equality.
\end{lemma}
\end{comment}

When the dimension of the complex space is an odd prime power, there is another construction which acheives the relative bound with equality. The following is the complex version of a set of real Grassmannian packings due to Calderbank, Hardin, Rains, Shor, and Sloane \cite{chrss}. For lack of another reference in the complex case, the details are included here.

Let $V := \ff{q}^n$, where $q = p^k$ and $p$ is an odd prime, and let $\{e_v: v \in V\}$ be the standard basis for $\cx^{q^n}$. Then define the $q^n \times q^n$ Pauli matrices
\begin{align*}
& X(a): e_v \mapsto e_{v+a}, \\
& Y(a): e_v \mapsto \om^{\tr(a^Tv)} e_v,
\end{align*}
where $\om$ is a $p$-th primitive root of unity. Note that $e_v$ is an eigenvalue for $Y(a)$ and $e_v^* := \sum_a \om^{\tr(a^Tv)} e_a$ is an eigenvalue for $X(a)$. Define the \defn{extraspecial Pauli group} $E$ to be generated by all $X(a), Y(a)$, and $\om I$; it has $pq^n$ elements, all of the form $\om^iX(a)Y(b)$, for $i \in \zz{p}$, $a,b \in V$. Its center is $Z(E) = \langle \om I \rangle$, and $\overline{E} := E/Z(E)$ is Abelian and therefore a vector space isomorphic to $V^2$ under the mapping 
\[
(a,b) \mapsto X(a)Y(b)/Z(E).
\]
The space $V^2$ has a nondegenerate alternating bilinear form (a \defn{symplectic} form), namely
\[
\ip{(a_1,b_1)}{(a_2,b_2)} := \tr(a_1^Tb_2 - a_2^Tb_1).
\]
It is not difficult to check that two elements in $E$, say $w^iX(a_1)Y(b_1)$ and $w^jX(a_2)Y(b_2)$, commute if and only if their images in $E/Z(E)$ satisfy 
\[
\ip{(a_1,b_1)}{(a_2,b_2)} = 0.
\]
Subspaces on which the symplectic form vanishes are called \defn{totally isotropic}. Therefore, a subspace $\overline{W}$ of $E/Z(E)$ is totally isotropic if and only if its preimage $W$ in $E$ is an Abelian subgroup.

We now use characters of subgroups of $E$ to define elements of $\gras{q^k}{q^n}$. Let $\overline{W}$ be a totally isotropic subspace of $E/Z(E)$ of dimension $n-k$, and let $W$ be the preimage of $\overline{W}$ in $E$. If $\chi: \overline{W} \rightarrow \cx$ is a character of $\overline{W}$, then $\chi': W \rightarrow \cx$ defined by
\[
\chi'(\om^iX(a)Y(b)) = \om^{-i}\chi(X(a)Y(b)/Z(E))
\]
is a character of $W$. Define a matrix 
\[
\Pi_\chi := \frac{1}{|W|} \sum_{g \in W} \chi'(g)g.
\]

\begin{lemma}
\label{lem:isosub}
If $\overline{W}$ is an $(n-k)$-dimensional totally isotropic subspace of $E/Z(E)$ and $\chi$ is a character of $\overline{W}$, then $\Pi_\chi$ is the projection matrix for a $q^k$-dimensional subspace of $\cx^{q^n}$ which is invariant under the action of $W$.
\end{lemma}

\begin{proof} 
It is not difficult to check that $\Pi_\chi$ is Hermitian and $\Pi_\chi^2 = \Pi_\chi$. It is also not difficult to check that $\Pi_\chi v$ is an eigenvector of $g \in W$ for any $v \in \cx^{p^n}$, so $\Pi_\chi$ is a projection matrix for an invariant subspace. The rank of $\Pi_\chi$ is the trace of $\Pi_\chi$, which can be computed as follows, after noting that the only elements of $E$ with non-zero trace are the multiples of the identity:
\[
\tr(\Pi_\chi) = \frac{1}{|W|} \sum_{g = \om^iI} \chi'(g)\tr(g) = \frac{1}{pq^{n-k}} \sum_{i = 1}^p \om^{-i} \tr(\om^iI) = q^k. \eqed
\]
\end{proof} 

In the construction that follows we require the $q$-binomial coefficients, defined as 
\[
\sbinom{n}{m}_q := \frac{(q^n-1)\ldots(q^{n-m+1}-1)}{(q^m-1)\ldots(q-1)}.
\]

\begin{theorem}
\label{thm:isosub}
For $0 \leq k \leq n-1$, let $S$ be the set of all $q^k$-dimensional invariant subspaces of the preimages $W$ of all $(n-k)$-dimensional totally isotropic subspaces $\overline{W}$ of $E/Z(E)$ (as described in Lemma \ref{lem:isosub}). Then $S$ is a $(n-k+1)$-distance set in $\gras{q^k}{q^n}$ of size 
\[
q^{n-k} \sbinom{n}{n-k}_q \; \prod_{i=k+1}^n(q^i + 1).
\]
\end{theorem}

\begin{proof}
For $j \in \{1,2\}$, let $\overline{W}_j$ be an isotropic subspace of $E/Z(E)$, let $W_j$ be its Abelian preimage in $E$, let $\chi_j$ be a character of $\overline{W}_j$, and let $\Pi_j := \Pi_{\chi_j}$ as in Lemma \ref{lem:isosub}. Then 
\begin{align*}
\tr(\Pi_1\Pi_2) &= \frac{1}{|W_1||W_2|} \sum_{g_1 \in W_1}\sum_{W_2 \in S_2} \chi_1'(g_1)\chi'_2(g_2) \tr(g_1g_2) \\
& = \frac{1}{|W_1||W_2|} \sum_{g_1 \in W_1 \cap W_2}\sum_{g_2 = \om^ig_1^{-1}} \chi_1'(g_1)\chi'_2(g_2) \tr(\om^iI) \\
%& = \frac{q^n}{|W_1||W_2|} \sum_{g_1 \in W_1 \cap W_2} \chi_1'(g_1) \sum_{g_2 = \om^ig_1^{-1}}\om^i \chi'_2(g_2)  \\
& = \frac{pq^n|W_1 \cap W_2|}{|W_1||W_2|} \text{ (or $0$, depending on $\chi'_1$ and $\chi'_2$}) \\
& = \frac{q^n|\overline{W_1 \cap W_2}|}{|\overline{W_1}||\overline{W_2}|} \text{ (or } 0). 
\end{align*}
Furthermore, any two distinct invariant subspaces from the same isotropic $\overline{W_j}$ are orthogonal. If $\overline{W_1} \neq \overline{W_2}$, then $\dim(\overline{W_1 \cap W_2}) \in \{0,1, \ldots, n-k-1\}$ and so $|\overline{W_1 \cap W_2}|$ takes $n-k$ possible values. It follows that $S$ is a $(n-k+1)$-distance set. To find the size of $S$, first note that the number of isotropic subspaces of dimension $n-k$ is (see \cite[Lemma 9.4.1]{bcn})
\[
\sbinom{n}{n-k}_q \; \prod_{i=k+1}^n(q^i + 1)
\]
and then note that each isotropic subspace produces $q^{n-k}$ invariant subspaces. \qed
\end{proof}

In the case $k = n-1$, Theorem \ref{thm:isosub} produces a $2$-distance set in $\gras{q^{n-1}}{q^n}$ of size $\frac{q(q^{2n}-1)}{q-1}$. The inner product values that occur are $\al = 0$ and $\be = q^{n-2}$: this construction acheives equality in the relative bound (Corollary \ref{cor:twobound}). In his thesis, Zauner \cite{zau} has a construction which has these same parameters (in fact, Zauner's construction is more general, as it also allows $q$ to be an even prime power). In the case $k = n-2$, we get a $3$-distance set in $\gras{q^{n-2}}{q^n}$ of size $\frac{q^2(q^{2n}-1)(q^{2n-2}-1)}{(q^2-1)(q-1)}$, with inner product values $\al = 0$, $\be = q^{n-4}$, and $\ga = q^{n-3}$.

There are many open questions regarding whether or not tightness in the bounds can be achieved; in particular, it is not known if there are any examples of subspaces achieving equality in the absolute bound (Corollary \ref{cor:abscodebnd1}) for $m > 1$. The smallest nontrivial case is a set of $16$ subspaces of dimension $2$ in $\cx^4$, with an inner product value of $\al = 14/15$. 

\section{Designs}
\label{sec:designs}

In this section, we introduce the concept of a \defn{complex Grassmannian $2$-design}. We give lower bounds for the size of a $t$-design and indicate the relationship between designs and codes.

Recall that $H_t(m,n)$ is the direct sum of the irreducible representations $H_\mu$ of $U(n)$ containing the zonal polynomials $Z_{\mu,a}$, where $\mu$ is an integer partition of size at most $t$ and length at most $m$. $H_t(m,n)$ may also be thought of as the symmetric polynomials of degree at most $t$ in the principle angles of pairs of subspaces in $\gras{m}{n}$. Since the zonal orthogonal polynomials $Z_{\mu,a}$ (with $\abs{\mu} \leq t$ and $\len(\mu) \leq m$) span $H_t(m,n)$ and are contained in $\Hom_t(n)$, it follows that $H_t(m,n)$ is a subspace of $\Hom_t(n)$. 

We call a finite subset $S \sbs \gras{m}{n}$ a \defn{$t$-design} if, for every polynomial $f$ in $H_t(m,n)$, 
\[
\frac{1}{|S|}\sum_{a \in S} f(a) = \int_{\gras{m}{n}} f(c) \; dc.
\]
In other words, the average of $f$ over $S$ is the same as the average of $f$ over the entire Grassmannian space. Recall that the average of $f$ over $\gras{m}{n}$ can be written as $\ip{1}{f}$: with this in mind we define an inner product for functions on $S$ as follows:
\[
\ip{f}{g}_S := \frac{1}{|S|} \sum_{a \in S} \conj{f(a)}g(a).
\]
Then $S$ is a $t$-design if $\ip{1}{f} = \ip{1}{f}_S$ for every $f \in H_t(m,n)$.
Equivalently, the zonal orthogonal polynomials $Z_{\mu,a}$ span $H_\mu$, so $S$ is a $t$-design if every $Z_{\mu,a}$ has the same averages over $S$ and $\gras{m}{n}$, where $\mu$ is a partition of at most $t$ into at most $m$ parts.

By way of example, consider Theorem \ref{thm:relcodebnd}. If $f = \sum_{\mu}c_{\mu} g_{\mu}$ and $c_{\mu} > 0$ for every $\abs{\mu} \leq t$, then equality in Theorem \ref{thm:relcodebnd} implies that $S$ is a $t$-design. 

For the purposes of quantum tomography applications, $1$- and $2$-designs play a special role (see \cite{grr}, as well as \cite{sco}). In those cases, there is a more explicit description of a $t$-design. 

\begin{lemma}
\label{lem:equivdes}
Let $S$ be a finite subset of $\gras{m}{n}$. Then $S$ is a $1$-design if and only if
\[
\frac{1}{|S|} \sum_{a\in S} P_a = \int_{\gras{m}{n}} P_a \; da = \frac{m}{n}I.
\]
Moreover, $S$ is a $2$-design if and only if 
\begin{equation}
\label{eqn:tswap}
\frac{1}{|S|} \sum_{a\in S} P_a \otimes P_a = \int_{\gras{m}{n}} P_a \otimes P_a \; da.
\end{equation}
\end{lemma}

Before proving Lemma \ref{lem:equivdes}, we note the integral on the RHS of equation \eqref{eqn:tswap} can be evaluated explicitly. Writing $P_a = \sum_{i=1}^m a_ia_i^*$ for some orthonormal basis $\{a_i\}$ of $a$, and letting $T$ denote the ``swap" operator $T: e_i \otimes e_j \mapsto e_j \otimes e_i$, the integral is obtained from Lemma 5.3 of \cite{rs1}:
\[
\int_{\gras{m}{n}} P_a \otimes P_a \; da = \frac{m}{n(n^2-1)}\left[(nm-1)I + (n-m)T\right]. 
\]

\begin{proof}
We prove the lemma by showing that $H_t = \Hom_t$ for $t \in \{1,2\}$; the result then follows by considering the polynomials of the form $a \mapsto (P_a)_{ij}$ in $\Hom_1$ and $a \mapsto (P_a)_{ij}(P_a)_{kl}$ in $\Hom_2$.

Recall that $H_t$ is contained in $\Hom_t$, so it suffices to show that the dimensions of the spaces are equal. When $t = 1$, we have $\dim(H_1) = \dim(\Hom_1) = n^2$, so $H_1 = \Hom_1$. When $t = 2$, recall that $\dim(H_2) = \binom{n^2}{2}$ (assuming $m > 1$), and the space of homogeneous degree-$2$ polynomials on the coordinates of $n \times n$ matrices has dimension $\binom{n^2+1}{2}$. However, $\Hom_2$ is the space of degree-$2$ polynomials on projection matrices, not general matrices. If $P_a$ is a projection matrix, then the degree-$2$ polynomial
\[
P_a \mapsto m\tr(AP_a^2) - \tr(P_a)\tr(AP_a)
\]
is identically zero for every $A$. There are $n^2$ linearly independent polynomials of that form for general $n \times n$ matrices; therefore,
\[
\dim(\Hom_2) = \binom{n^2+1}{2} - n^2 = \binom{n^2}{2}.
\]
Thus $H_2 = \Hom_2$. \qed
\end{proof}

We now consider bounds for $t$-designs. The following is the \defn{absolute bound}.

\begin{lemma}
If $S$ is a $t$-design, then
\[
|S| \geq \dim(H_{\floor{t/2}}(m,n)).
\]
\end{lemma}

\begin{proof}
Let $\{e_1,\ldots,e_N\}$ be an orthonormal basis for $H_{\floor{t/2}}$. Since $e_i$ is a symmetric polynomial in the eigenvalues, so is $\conj{e_i}e_j$. It follows from the unique decomposition of $L^2(\gras{m}{n})$ that $\conj{e_i}e_j$ is in $H_{2\floor{t/2}}$ and therefore in $H_t$. If $S$ is a $t$-design, and $\conj{e_i}e_j$ is in $H_t$, then
\[
\ip{e_i}{e_j} = \ip{1}{\conj{e_i}e_j} = \ip{1}{\conj{e_i}e_j}_S = \ip{e_i}{e_j}_S,
\]
whence it follows that $\{e_1,\ldots,e_l\}$ are orthogonal as functions of $S$ (a space of dimension $|S|$). \qed
\end{proof}

If equality holds, then the basis for $H_{t/2}(m,n)$ is also a basis for the functions on $S$. There is also a \defn{relative bound}.

\begin{theorem}
\label{thm:reldesbnd}
Let $f(x_1,\ldots,x_m) \in \re[x]$ be a symmetric polynomial such that $f = \sum_{\mu}c_{\mu} Z_{\mu}$, where $Z_{\mu}$ is a zonal polynomial for the Grassmanian space, and $c_0 > 0$. Furthermore, suppose $S$ is a $t$-design such that $f_a(b) = f(y_1(a,b),\ldots,y_m(a,b)) \geq 0$ for every $a \neq b$ in $S$, and $c_\mu \leq 0$ for every $\abs{\mu} > t$. Then
\[
|S| \geq \frac{f(1,\ldots,1)}{c_0}.
\]
\end{theorem}

\begin{proof}
Let $f_a$ be the zonal polynomial of $f$ at $a$, so that 
$f_a(b) \geq 0$ for $b \neq a$. Summing over all $b \in S$,
\[
|S|\ip{1}{f_a}_S \geq f_a(a) = f(1,\ldots,1).
\]
Again averaging over all $a \in S$,
\begin{align*}
f(1,\ldots,1) & \leq \sum_{a \in S} \ip{1}{f_a}_S \\
& = \sum_{a \in S} \sum_\mu c_\mu \ip{1}{Z_{\mu,a}}_S \\
& = \sum_\mu c_\mu \sum_{a \in S} \ip{1}{Z_{\mu,a}}_S.
\end{align*}
Since $S$ is a $t$-design, the inner sum is zero for $\abs{\mu} \leq t$ ($\mu \neq 0$). For $\abs{\mu} > t$, the inner sum is nonnegative (by Lemma \ref{lem:possum}) and $c_\mu \leq 0$. Therefore,
\begin{align*}
f(1) & \leq c_0 \sum_{a \in S} \ip{1}{Z_{0,a}}_S \\
& = c_0 |S|. \eqed
\end{align*} 
\end{proof}

If equality holds, then we have $f(a,b) = 0$ for every $a \neq b$ in $S$. That is, $S$ is an $f$-code. Furthermore, for every $\abs{\mu} > t$, we have either $c_\mu = 0$ or $\sum_{a \in S} Z_{\mu,a} = 0$.

As with classical codes and designs, the case where $S$ is both a $f$-code and a $t$-design is of particular interest, as the size of the set can be determined exactly. Combining Theorems \ref{thm:relcodebnd} and \ref{thm:reldesbnd} gives the following.

\begin{theorem}
\label{thm:absdescode}
Suppose $S$ is an $f$-code for $f = \sum_\mu c_\mu Z_\mu$, where $c_\mu \geq 0$, and $S$ is also a $t$-design for $t \geq \deg(f)$. Then
\[
|S| = \frac{f(1,1,\ldots,1)}{c_0}.
\]
\end{theorem}

Consider the following polynomial in $H_t(m,n)$: 
\[
Z_t := \sum_{\substack{\abs{\mu} \leq t \\ \len(\mu) \leq m}} Z_\mu.
\]
This polynomial satisfies $\ip{Z_{t,a}}{f} = f(a)$ for every $f \in H_t(m,n)$. Taking $f = Z_t$ in Theorem \ref{thm:absdescode}, we get:

\begin{corollary}
If $S$ is a $Z_t$-code and a $2t$-design, then 
\[
|S| = \dim(H_t(m,n)).
\] 
\end{corollary}

\begin{theorem}
\label{thm:twothree}
Any two of the following imply the third:
\begin{itemize}
\item $S$ is an $f$-code, where $\deg(f) = t$;
% $t$-distance set;
\item $S$ is a $2t$-design;
\item $|S| = \dim(H_t(m,n))$.
\end{itemize}
\end{theorem}

\begin{proof}
Suppose $S$ is a $f$-code with $|S| = \dim(H_t)$. Since equality holds in Corollary \ref{cor:abscodebnd}, the polynomials $f_a$ are a basis for $H_t$. However, we have
\[
\ip{Z_{t,a}}{f_b} = f_b(a) = \begin{cases}
0, & b \neq a; \\
f(1,1,\ldots,1), & b = a. \\
                             \end{cases}
\]
Thus $\{Z_{t,a}\}$ is a dual basis for $H_t$ and each $Z_{t,a}$ is a multiple of $f_{t,a}$. Now consider the averages $\ip{Z_{t,a}}{f_b}_S$: since $f_a(b) = Z_{t,a}(b) = 0$ for $b \neq a$, we get
\[
\ip{Z_{t,a}}{f_b}_S = \begin{cases}
0, & b \neq a; \\
f(1,1,\ldots,1), & b = a. \\
                             \end{cases}
\]
Thus we have
\[
\ip{1}{Z_{t,a}f_b}_S = \ip{\conj{Z_{t,a}}}{f_b}_S = \ip{\conj{Z_{t,a}}}{f_b} = \ip{1}{Z_{t,a}f_b}
\]
for the bases $\{Z_{t,a}\}$ and $\{f_b\}$. But the set $\{Z_{t,a}f_b\}$ spans $H_{2t}(n)$, so $S$ is a $2t$-design. 

Conversely, suppose $S$ is a $2t$-design with $|S| = \dim(H_t)$, and let $f$ annihilate of the angle set of $A$. Since $H_t$ spans the functions on $|S|$, each $f_a$ is in $H_t$ and is therefore a polynomial of degree $t$. Thus $f$ has degree $t$. \qed
\end{proof}

The simplest case of Theorem \ref{thm:twothree} is when $t = 1$: in this case, $S$ is a $1$-distance set and a $2$-design of size $n^2$. Moreover, $S$ is a $Z_1$-code, and $Z_1$ is the annihilator of $\tfrac{m(mn-1)}{n^2-1}$. Thus the inner product between every two distinct subspaces is $\al = \tfrac{m(mn-1)}{n^2-1}$.

\begin{comment}
If Theorem \ref{thm:twothree} holds, then the annihilator of $S$ is simply $Z_k$, and so $Z_k$ has the form of an annihilator
\[
Z_{k,a}(b) = \prod_{i=1}^k (\tr(P_aP_b) - \al_i).
\]
In the case for $k=1$, where we have
\[
Z_{1,a}(b) = \frac{n(n^2-1}{m(n-m)}\Big[\tr(P_aP_b)- \frac{m(mn-1)}{n^2-1}\Big],
\]
so $Z_1$ is the annihilator of $\al = \tfrac{m(mn-1)}{n^2-1}$. For $k = 2$, we have
\begin{align*}
Z_k(y) & = Z_0(y) + Z_1(y) + Z_2(y) + Z_{1,1}(y) \\
& =  \frac{n^2(n^2-1)(n+2)(n+3)}{4(n-m+1)(n-m)}X^*_2 + \frac{n^2(n^2-1)(n-2)(n-3)}{4(n-m)(n-m-1)}X^*_{1,1} \\ 
& \qquad + \frac{n(n^2-1)(m^2n^2-mn^3+4mn-n^2+m^2-1)}{(n-m)(n-m+1)(n-m-1)}X^*_1 \\
& \qquad + \frac{n(n^2-1)(m^3n-m^2n^2+mn+2m^2-2)}{2(n-m)(n-m+1)(n-m-1)}.
\end{align*}
{\bf Here we are assuming the correct normalizations.} If $Z_2$ is an annihilator, say of $\al$ and $\be$, then
\begin{align*}
Z_2(y) & = c(\sum_i y_i - \al)(\sum_i y_i - \be) \\
& = c\big[(\sum_i y_i^2 + \sum_{i<j}y_iy_j) + \sum_{i<j}y_iy_j - (\al+\be)\sum_i y_i + \al\be\big] \\
& = c\big[{m+1 \choose 2}X_2^* + {m \choose 2}X_{1,1}^* - m(\al+\be)X_1^* + \al\be\big].
\end{align*}
Two equations for $Z_2$ can be equal only if ratios of the coefficients of $X_2^*$ and $X_{1,1}^*$ are the same:
\[
\frac{{m+1 \choose 2}}{ {m \choose 2}} = \frac{\frac{n^2(n^2-1)(n+2)(n+3)}{4(n-m)(n-m+1)}}{\frac{n^2(n^2-1)(n-2)(n-3)}{4(n-m)(n-m-1)}},
\]
which simplifies to 
\[
n^2 - 5mn + 5m^2+1 = 0.
\]
for $n \neq 0$. The last equation is satisfied only for certain particular integer values such as $(n,m) = (7,2),(18,5),(47,13)$. 
\end{comment}

\section{Association schemes}
\label{sec:assoc}

As Theorem \ref{thm:twothree} indicates, sets of Grassmannian subspaces which reach equality in the Delsarte bounds have a great deal of structure. In this section, we show that---much like spherical codes and spherical designs---these sets are often endowed with the structure of an association scheme. 

Let $S$ be an $f$-code with a finite number of distinct sets of principal angles $y = (y_1,\ldots,y_m)$. Denote the set of $y$'s that occur by $\scr{Y}$. For each $y \in \scr{Y}$, define a $|S| \times |S|$ matrix as follows:
\[
A_y(a,b) := \begin{cases}
1, & a,b \mbox{ have principal angles } y; \\
0, & \mbox{otherwise.}
           \end{cases}
\]
Each $A_y$ is a symmetric $\{0,1\}$-matrix. Furthermore, each pair $(a,b)$ has some principal angle $y$, so $\sum_{y \in \scr{Y}} A_y = J$, where $J$ is the all-ones matrix. If $y_0 := (1,\ldots,1)$ denotes the trivial principal angles set, then $A_0 := A_{y_0}$ is the identity matrix. We will call the $A_y$ matrices \defn{Schur idempotents}, as they are idempotent under Schur multiplication, defined as follows:
\[
(A \circ B)_{ij} := A_{ij} B_{ij}.
\]
Under certain conditions, these Schur idempotents form an association scheme. 

For each integer partition $\mu$ and corresponding zonal polynomial $Z_\mu$, define an $|S| \times |S|$ matrix as follows:
\[
E_{\mu}(a,b) := \frac{1}{|S|}Z_\mu(a,b).
\]
Each $E_\mu$ is also symmetric and in the span of $\{A_y\}_{y \in \scr{Y}}$:
\[
E_\mu = \frac{1}{|S|}\sum_{y \in \scr{Y}} Z_\mu(y) A_y.
\]
In particular, $E_0$ is a scalar multiple of $J$. When $\{A_y\}_{y \in \scr{Y}}$ forms an association scheme, the matrices $E_\mu$ are the scheme's idempotents.

\begin{lemma}
If $S$ is a $2t$-design, then $\{E_\mu\}_{\abs{\mu} \leq t, \len(\mu) \leq m}$ are a set of orthogonal idempotents.
\label{lem:orthogidem}
\end{lemma}

\begin{proof}
Suppose $\abs{\mu} = i$ and $\abs{\la} = j$, with $i,j \leq t$. Then
\begin{align*}
(E_\mu E_\la)_{a,b} & = \frac{1}{|S|^2}\sum_{c \in S} Z_\mu(a,c)Z_\la(c,b) \\
& = \frac{1}{|S|}\ip{Z_{\mu,a}}{Z_{\la,b}}_S.
\end{align*}
Since $Z_{\mu,a}$ and $Z_{\mu,b}$ are in $H_t$, their product is in $H_{2t}$. Now $S$ is a $2t$-design, so the average of $Z_{\mu,a}Z_{\la,b}$ over $S$ is the same as the average over $\gras{m}{n}$. But
\[
\ip{Z_{\mu,a}}{Z_{\la,b}} = \de_{\la,\mu}Z_\mu(a,b),
\]
and so we find that $E_\mu E_\la = \de_{\la,\mu}E_\mu$. \qed
\end{proof}

More generally, if $\abs{\mu} = i$ and $\abs{\la} = j$, and $S$ is a $(i+j)$-design, then $E_\mu$ and $E_\la$ are orthogonal. 

Now suppose $S$ is a $2t$-design. By the previous lemma $\{E_\mu\}_{\abs{\mu} \leq t}$ are linearly independent, and clearly the matrices $\{A_y\}_{y \in \scr{Y}}$ are also linearly independent. If $|\scr{Y}|$ equals the number of partitions of at most $t$ (into at most $m$ parts), then the span of $\{A_y\}_{y \in \scr{Y}}$ and $\{E_\mu\}_{\abs{\mu} \leq t}$ are the same. Since $\{E_\mu\}_{\abs{\mu} \leq t}$  is closed under multiplication, so too is the span of $\{A_y\}_{y \in \scr{Y}}$, and so we have an association scheme. 

\begin{corollary}
\label{cor:assoc}
Let $S$ be a $2t$-design in $\gras{m}{n}$ with principal angle set $\scr{Y}$. If $|\scr{Y}|$ is equal to the total number of partitions of $0,1,\ldots,t$ into at most $m$ parts, then $\{A_y\}_{y \in \scr{Y}}$ is an association scheme.
\end{corollary}

\begin{lemma}
Let $S$ be a $2t$-design in $\gras{m}{n}$ with principal angle set $\scr{Y}$ such that $|\scr{Y}|$ is the total number of partitions of $0,1,\ldots,t$ into at most $m$ parts. Then $\{E_\mu\}_{\abs{\mu} \leq t, \len(\mu) \leq m}$ are the idempotents of the scheme $\{A_y\}_{y \in \scr{Y}}$.
\end{lemma}

\begin{proof}
Since $E_\mu = \frac{1}{|S|}\sum_{y \in \scr{Y}} Z_\mu(y) A_y$, we see that the matrix $[Z_\mu(y)]$ is the transition matrix between the two bases of the association scheme and is therefore invertible. It follows that for each $y_i$ in $\scr{Y}$, some linear combination of the rows $Z_\mu$ forms a homogeneous degree-$t$ polynomial $g_i$ such that $g_i(y_j) = \de_{ij}$. (Conversely, if such $g_i$ polynomials exist, then $[Z_\mu(y)]$ is invertible.) Then
\begin{align*}
(A_iE_\mu)_{a,b} & = \frac{1}{|S|}\sum_{c: y(a,c) = y_i} Z_\mu(c,b) \\
& = \ip{g_{i,a}}{Z_{\mu,b}}_S \\
& = \ip{g_{i,a}}{Z_{\mu,b}}.
\end{align*}
Now write $g_i = \sum_{\abs{\la} \leq t} c_{i,\la} Z_\la$, so that
\[
\ip{g_{i,a}}{Z_{\mu,b}} = \sum_{\abs{\la} \leq t} c_{i,\la} \ip{Z_{\la,a}}{Z_{\mu,b}} = c_{i,\mu} Z_\mu(a,b). 
\]
Thus $A_iE_\mu = c_{i,\mu} E_\mu$ for some $c_{i,\mu}$. \qed
\end{proof}

By way of example, let $t = 1$, and suppose $S$ is a $2$-design with only one nontrivial principal angle set (and one trivial one, for a total of two). The number of partitions of at most $1$ is also two ($\mu = 0$ and $\mu = (1)$), so by Corollary \ref{cor:assoc} we have an association scheme. In this case the scheme is the trivial one, namely $\{I, J-I\}$.

As another example of an association scheme obtained from principal angles, consider the collection of subspaces in $\gras{n/2}{n}$ from Theorem \ref{thm:pauli}. This collection has four distinct sets of principal angles: 
\begin{align*}
y & = (1,\ldots,1) \quad (\text{trivial principal angles}), \\
y & = (0,\ldots,0) \quad (\text{angles between $a$ and $a^\perp$}), \\
y & = (\underbrace{1,\ldots,1}_{n/4},\underbrace{0,\ldots,0}_{n/4}), \\ 
y & = (\tfrac{1}{2},\ldots,\tfrac{1}{2}).
\end{align*}
While $|\scr{Y}| = 4$ is the number of partitions of at most $2$ ($\mu = 0$, $\mu = (1)$, $\mu = (1,1)$ and $\mu = (2)$), the hypotheses of Corollary \ref{cor:assoc} are not satisfied because the subspaces do not form a $4$-design. Nevertheless, it is easy to verify computationally that this collection does give a $3$-class association scheme.

We may define a coarser set of relations on an $f$-code $S$ using the sums of principal angles---the inner products of the projection matrices---instead of the principal angles themselves. Let $\scr{A}$ denote the set of nontrivial inner product values that occur in $S$, so $S$ is an $\scr{A}$-code. For $\al \in \scr{A}$ let $A'_\al$ be the $|S| \times |S|$ matrix defined as follows:
\[
A'_\al(a,b) := \begin{cases}
1, & \tr(P_aP_b) = \al, \\
0, & \mbox{otherwise.}
           \end{cases}
\]
Also define $A'_m := I$ for the identity relation. Clearly each $A'_\al$ is in the span of $\{A_y: y \in \scr{Y}\}$; in fact
\[
A'_\al = \sum_{y \in \scr{Y}:\; \sum y_i = \al} A_y.
\]
In particular, $A'_m = A_0 = I$, and if $0$ is in $\scr{A}$, then $A'_0 = A_{(0,\ldots,0)}$. As before, the matrices are Schur idempotents and sum to $J$. Next we need the corresponding idempotents. For each $i \in \{0,\ldots,t\}$, define $E'_i$ as follows:
\[
E'_i := \sum_{\abs{\mu} = i} E_\mu.
\]
This implies that $E'_0 = J/|S|$ and $E'_i(a,b) = (Z_i(a,b) - Z_{i-1}(a,b))/|S|$ for $i > 0$. As in Lemma \ref{lem:orthogidem}, if $S$ is a $2t$-design, then $\{E'_i: i \leq t\}$ is a set of orthogonal idempotents, and if $S$ is a $(2t-1)$-design, then $\{E'_i: i \leq t\}$ are linearly independent.

Clearly $E'_i$ is in the span of $\{A_y: y \in \scr{Y}\}$, since each $E_\mu$ is in that span. But suppose $Z_i(y)$ is the annihilator polynomial of some $i$-distance set, so it is a only function of $\sum_i y_i$: then in fact $E'_i$ is in the span of $\{A'_\al: \al \in \scr{A}\}$. If $Z_i(y)$ is an annihilator for sufficiently many $i$, then $\{E'_i: 0 \leq i \leq t\}$ and $\{A'_\al: \al \in \scr{A}\cup\{m\}\}$ span the same set, and that set is closed under multiplication.

\begin{corollary}
\label{cor:sumassoc}
Let $S$ be a $2t$-design that is also an $\scr{A}$-code in $\gras{m}{n}$. If $|\scr{A}| \leq t$, and $Z_i(y)$ is an annihilator polynomial for each $i \leq t$, then $\{A'_\al: \al \in \scr{A}\cup\{m\}\}$ is an association scheme.
\end{corollary}

In fact, these hypotheses can be weakened.

\begin{theorem}
\label{thm:weakscheme}
Let $S$ be a $(2t-2)$-design that is also an $\scr{A}$-code in $\gras{m}{n}$. If $|\scr{A}| = t$, and $Z_i(y)$ is an annihilator for each $0 \leq i \leq t-1$, then $\{A'_\al: \al \in \scr{A}\cup\{m\}\}$ is an association scheme.
\end{theorem}

\begin{proof}
Since $S$ is a $2(t-1)$-design, the idempotents $\{E'_i: 0 \leq i \leq t-1\}$ are linearly independent. We claim that $I$ is also linearly independent from $\{E'_i: 0 \leq i \leq t-1\}$. For, if $I = \sum_{i=0}^{t-1} c_iE'_i$, then the off-diagonal entries of $I$ are functions of a polynomial of degree at most $t-1$ in $\sum_j y_j$, namely 
\[
\frac{1}{|S|}\left(c_0 + \sum_{i=1}^{t-1}c_i(Z_i(y)-Z_{i-1}(y)) \right).
\] 
But all off-diagonal entries are $0$, implying that the polynomial has $t$ roots in $\sum_i y_i$, a contradition. So $\{E'_i: 0 \leq i \leq t-1\} \cup \{I\}$ is linearly independent and therefore spans $\{A'_\al: \al \in \scr{A}\cup\{m\}\}$. Since it is closed under multiplication, we have an association scheme. \qed
\end{proof}

By way of example, suppose $t=2$ in Theorem \ref{thm:weakscheme}. Note that $Z_0(y)$ and $Z_1(y)$ are always annihilators. It follows that if $S$ is a $2$-design, and the inner product set $\scr{A} = \{\tr(P_aP_b): a \neq b \in S\}$ contains exactly two distinct values, then $\{A'_\al: \al \in \scr{A}\cup\{m\}\}$ is a $2$-class association scheme. 

\begin{corollary}
Let $S$ be a $(2t-2)$-design and an $\scr{A}$-code in $\gras{m}{n}$ such that $|\scr{A}| = t$ and $Z_i(y)$ is an annihilator for $i \leq t-1$. Then the idempotents of the scheme $\{A'_\al: \al \in \scr{A}\cup\{m\}\}$ are $E'_0,\ldots,E'_{t-1}$, and $J - \sum_{i=0}^{t-1}E'_i$.
\end{corollary}

\begin{proof}
Let $f_\al$ denote the annihilator polynomial of $\scr{A}\backslash\{\al_0,\al\}$, normalized so that $f_\al(\al) = 1$. Then $f_\al$ is a polynomial of degree $t-1$ in $\sum_i y_i$, and the corresponding zonal polynomial $f_{\al,a}$ is in $H_{t-1}(n)$. Writing $P_i := Z_i - Z_{i-1} = \sum_{\abs{\mu} = i} Z_\mu$, we have
\begin{align*}
(A'_\al E'_i)_{a,b} & = \frac{1}{|S|}\sum_{\tr(P_aP_c) = \al} P_i(\tr(P_cP_b)) \\
& = \ip{f_{\al,a}}{P_{i,b}}_S - \frac{f_\al(m)}{|S|}P_i(\tr(P_aP_b)) \\
& = \ip{f_{\al,a}}{P_{i,b}} - \frac{f_\al(m)}{|S|}P_i(\tr(P_aP_b)).
\end{align*}
Now decomposing into its degrees as $f_\al = \sum_i c_{\al,i} P_i$, we get
\begin{align*}
(A'_\al E'_i)_{a,b} & = c_{\al,i}\ip{P_{i,a}}{P_{i,b}} - \frac{f_\al(m)}{|S|}P_i(\tr(P_aP_b)) \\
& = c_{\al,i}P_i(\tr(P_aP_b)) - \frac{f_\al(m)}{|S|}P_i(\tr(P_aP_b)) \\
& = (c_{\al,i}|S| - f_\al(m))(E'_i)_{a,b}.
\end{align*}
Thus $A'_\al E'_i = \la_{\al,i}E'_i$ for some constant $\la_{\al,i}$. \qed
\end{proof}

\section{Acknowledgements}

The author would like to thank Martin R\"otteler, Chris Godsil, Bill Martin, and Barry Sanders for their helpful discussions. This work was funded by NSERC and MITACS.

\bibliographystyle{agbib}
\bibliography{quantum}

\begin{thebibliography}{10}

\bibitem{aru1}
{\sc D.~Agrawal, T.~J. Richardson, and R.~L. Urbanke}, {\em Multiple-antenna
  signal constellations for fading channels}, IEEE Trans. Inform. Theory, 47
  (2001), 2618--2626.

\bibitem{bac1}
{\sc C.~Bachoc}, {\em Linear programming bounds for codes in {G}rassmannian
  spaces}, IEEE Trans. Inform. Theory, 52 (2006), 2111--2125.

\bibitem{bcn1}
{\sc C.~Bachoc, R.~Coulangeon, and G.~Nebe}, {\em Designs in {G}rassmannian
  spaces and lattices}, J. Algebraic Combin., 16 (2002), 5--19.

\bibitem{bor}
{\sc K.~B{\"o}r{\"o}czky, Jr.}, {\em Finite {P}acking and {C}overing}, vol.~154
  of Cambridge Tracts in Mathematics, Cambridge University Press, Cambridge,
  2004.

\bibitem{bcn}
{\sc A.~E. Brouwer, A.~M. Cohen, and A.~Neumaier}, {\em Distance-{R}egular
  {G}raphs}, Springer-Verlag, Berlin, 1989.

\bibitem{bump}
{\sc D.~Bump}, {\em Lie {G}roups}, vol.~225 of Graduate Texts in Mathematics,
  Springer-Verlag, New York, 2004.

\bibitem{chrss}
{\sc A.~R. Calderbank, R.~H. Hardin, E.~M. Rains, P.~W. Shor, and N.~J.~A.
  Sloane}, {\em A group-theoretic framework for the construction of packings in
  {G}rassmannian spaces}, J. Algebraic Combin., 9 (1999), 129--140.

\bibitem{chs}
{\sc J.~H. Conway, R.~H. Hardin, and N.~J.~A. Sloane}, {\em Packing lines,
  planes, etc.: packings in {G}rassmannian spaces}, Experiment. Math., 5
  (1996), 139--159.

\bibitem{cs1}
{\sc J.~H. Conway and N.~J.~A. Sloane}, {\em Sphere {P}ackings, {L}attices and
  {G}roups}, vol.~290 of Grundlehren der Mathematischen Wissenschaften
  [Fundamental Principles of Mathematical Sciences], Springer-Verlag, New York,
  second~ed., 1993.

\bibitem{del1}
{\sc P.~Delsarte}, {\em An algebraic approach to the association schemes of
  coding theory}, Philips Res. Rep. Suppl.,  (1973), vi+97.

\bibitem{dgs}
{\sc P.~Delsarte, J.~M. Goethals, and J.~J. Seidel}, {\em Bounds for systems of
  lines, and {J}acobi polynomials}, Philips Res. Rep.,  (1975), 91--105.

\bibitem{fh1}
{\sc W.~Fulton and J.~Harris}, {\em Representation {T}heory}, Springer-Verlag,
  New York, 1991.

\bibitem{grr}
{\sc C.~Godsil, M.~R{\"o}tteler, and A.~Roy}, {\em Mutually unbiased
  subspaces}, {i}n preparation.

\bibitem{gr1}
{\sc C.~Godsil and A.~Roy}, {\em Mutually unbiased bases, equiangular lines,
  and spin models}, {t}o appear in European Journal of Combinatorics,  (2007).

\bibitem{gw1}
{\sc R.~Goodman and N.~R. Wallach}, {\em Representations and {I}nvariants of
  the {C}lassical {G}roups}, vol.~68 of Encyclopedia of Mathematics and its
  Applications, Cambridge University Press, Cambridge, 1998.

\bibitem{hel}
{\sc S.~Helgason}, {\em Groups and {G}eometric {A}nalysis}, vol.~113 of Pure
  and Applied Mathematics, Academic Press Inc., Orlando, FL, 1984.

\bibitem{jc1}
{\sc A.~T. James and A.~G. Constantine}, {\em Generalized {J}acobi polynomials
  as spherical functions of the {G}rassmann manifold}, Proc. London Math. Soc.
  (3), 29 (1974), 174--192.

\bibitem{kha}
{\sc M.~Khatirinejad}, {\em On {W}eyl-{H}eisenberg orbits of equiangular
  lines}, Journal of Algebraic Combinatorics,  (2007).

\bibitem{rbsc}
{\sc J.~Renes, R.~Blume-Kohout, A.~J. Scott, and C.~M. Caves}, {\em Symmetric
  informationally complete quantum measurements}, J. Math. Phys., 45 (2004),
  2171.

\bibitem{rs1}
{\sc A.~Roy and A.~J. Scott}, {\em Weighted complex projective 2-designs from
  bases: optimal state determination by orthogonal measurements}, J. Math.
  Phys., 48 (2007), 072110.

\bibitem{sco}
{\sc A.~J. Scott}, {\em Tight informationally complete quantum measurements},
  J. Phys. A, 39 (2006), 13507--13530.

\bibitem{sep}
{\sc M.~R. Sepanski}, {\em Compact {L}ie {G}roups}, vol.~235 of Graduate Texts
  in Mathematics, Springer, New York, 2007.

\bibitem{stan}
{\sc R.~P. Stanley}, {\em Enumerative {C}ombinatorics. {V}ol. 2}, vol.~62 of
  Cambridge Studies in Advanced Mathematics, Cambridge University Press,
  Cambridge, 1999.

\bibitem{wong}
{\sc Y.-c. Wong}, {\em Differential geometry of {G}rassmann manifolds}, Proc.
  Nat. Acad. Sci. U.S.A., 57 (1967), 589--594.

\bibitem{zau}
{\sc G.~Zauner}, {\em Quantendesigns}, PhD thesis, University of Vienna, 1999.

\end{thebibliography}

\end{document}